\documentclass{article}
\usepackage{amssymb,amsmath,amsthm}
\usepackage{epsfig}
\begin{document}
\newtheorem{prop}{Proposition}
\newtheorem{lem}[prop]{Lemma}
\newtheorem{rem}{Remark}
\def\E{\mathbb{E}}
\def\P{\mathbb{P}}
\def\R{\mathbb{R}}
\def\Z{\mathbb{Z}}
\def\H{\mathcal H}
\def\O{\mathcal O}
\def\scr{\scriptstyle}
\def\text{\textstyle}
\def\om{\omega}
\def\del{\delta}
\def\na{\nabla}
\def\hb{{\bf h}}
\def\hh{\hat h}
\def\d{\mathrm d}
\def\e{\mathrm e}
\def\i{\mathrm i}
\def\12{\frac{1}{2}}
\def\beq{\begin{equation}}
\def\eeq{\end{equation}}
\def\beqa{\begin{eqnarray}}
\def\eeqa{\end{eqnarray}}
\title{Space-time correlations of a Gaussian interface}
\author{Fran\c cois Dunlop\cr
{\small\it Laboratoire de Physique Th\'eorique et Mod\'elisation
{\rm (CNRS -- UMR 8089)}}\cr
{\small\it Universit\'{e} de Cergy-Pontoise, 95302 Cergy-Pontoise,
France}
}
\maketitle
\begin{abstract}
The serial harness introduced by Hammersley \cite{[H]} is equivalent, in the
Gaussian case, to the Gaussian Solid-On-Solid interface model with
parallel heat bath dynamics. Here we consider sub-lattice parallel
dynamics, and give exact results about relaxation dynamics, based on the
equivalence to the infinite time limit of a time periodic random field.
We also give a numerical comparison to the harness process in
continuous time studied by Hsiao \cite{[Hs]} 
and by Ferrari, Niederhauser and Pechersky \cite{[FN],[FNP]}. 
\end{abstract}


\section{Introduction}
\label{intro}
Let $L$ be a positive even integer and let the initial condition
$h^0=\{h^0_i:i\in\Z/L\Z\}$ be distributed according to the
un-normalized measure 
\begin{equation}\label{eqm}
\mu(\d h^0)=\prod_i\e^{^{\text 
-\frac{1}{2}\bigl(\,h^0_{i+1}-h^0_i\bigr)^2}}\d h^0_i
\end{equation}
where $\d h^0_i$ is the Lebesgue measure over $\R$. The index $i$ runs over
$\Z/L\Z$, which corresponds to periodic boundary conditions. The
measure (\ref{eqm}) may be considered as a finite volume Gibbs measure with
Hamiltonian 
\[
H(h^0)=\frac{1}{2}\sum_i\bigl(\,h^0_{i+1}-h^0_i\bigr)^2
\]
The corresponding sub-lattice parallel heat bath dynamics is defined by
\begin{multline}\label{slpd}
\P\bigl(\,\d h^t\,|\,h^{t-1}\,\bigr)=\prod_{i+t\ {\rm even}}
\e^{^{\text-\bigl(\,h^t_i-\12(h^{t-1}_{i-1}+h^{t-1}_{i+1})
\,\bigr)^2}}\,.\\
.\prod_{i+t\ {\rm odd}}\del(h^t_i-h^{t-1}_i)\prod_i\d h^t_i\bigg/{\rm norm.}
\end{multline}
where the normalization of the probability is a finite constant, independent 
of $h^{t-1}$. 
The stochastic process defined by (\ref{slpd}) is intermediate between
Hammersley's original serial harness \cite{[H]} and the harness process in
continuous time \cite{[Hs],[FN],[FNP]}. Various sub-lattice parallel stochastic
dynamics for interface models have been studied, e.g. in \cite{[CDFG],[Du]},
showing a closer similarity with continuous time dynamics than with
fully parallel dynamics. 

The heat bath dynamics leaves invariant the Gibbs measure which motivated it:
\begin{equation}
\int\mu(\d h^{t-1})\P\bigl(\,\d h^t\,|\,h^{t-1}\,\bigr)=\mu(\d h^t)
\end{equation}
As the initial condition $h^0$ is already distributed with the measure
$\mu$, we have a stationary problem. 
Our main result is a computation of space-time correlations, in the 
thermodynamic limit $L\to\infty$. The correlation function of two
space gradients at time and space separation $(2t,j)$ will be denoted
$g_{11}(t,j)$, the correlation function of two
time gradients at time and space separation $(2t,j)$ will be denoted
$g_{22}(t,j)$. The time separation $2t$ corresponds to $t$ updates at
each site between the two events:
\beq\label{g11}
g_{11}(t,j)=\lim_{L\to\infty}\E (h^{2t}_{j+2}-h^{2t}_j)(h^0_2-h^0_0)\,,\quad t\ge0
\eeq
\beq\label{g22}
g_{22}(t,j)=\lim_{L\to\infty}\E (h^{2t+2}_j-h^{2t}_j)(h^2_0-h^0_0)\,,\quad t\ge0
\eeq
and similarly
\beq\label{g12}
g_{12}(t,j)=\lim_{L\to\infty}\E (h^{2t}_{j+1}-h^{2t}_{j-1})(h^2_0-h^0_0)\,,\quad 
t\ge1
\eeq
\beq\label{g21}
g_{21}(t,j)=\lim_{L\to\infty}\E (h^{2t}_0-h^{2t-2}_0)(h^0_{j+1}-h^0_{j-1})\,,\quad 
t\ge1
\eeq
\begin{prop}\label{prop1}
Let $h^{[0,t]}$ be distributed according to (\ref{eqm})(\ref{slpd}),
for each $t\in\Z_+$. Then for each $t\in\Z_+$ and $j\in\Z$ the limits 
(\ref{g11})(\ref{g22})(\ref{g12})(\ref{g21}) exist and satisfy
\beqa\label{tj}
g_{11}(t,j)&=&\frac{2^{-2t+1}(2t)!}{(t-\frac{j}{2})!(t+\frac{j}{2})!}
\qquad{\rm if}\quad 0\leq j\leq 2t\,,\\
g_{11}(t,-j)&=&-g_{11}(t,j) \qquad{\rm if}\quad j\ge1\label{sym}\\
g_{11}(t,j)&=&0\qquad{\rm if}\quad |j|>2t\ge2.\label{cause}
\eeqa
\beq\label{g2211}
g_{22}(t,j)=-\frac{1}{4}\Bigl[g_{11}(t-1,|j|)-g_{11}(t,|j|)\Bigr]
\qquad{\rm if}\quad t\ge1
\eeq
\beqa
g_{12}(t,j)&=&-g_{12}(t,-j)=-g_{21}(t,j)=g_{21}(t,-j)\label{g1221}\\
&=&g_{11}(t-1,|j+1|)-g_{11}(t-1,|j-1|)\label{g1211}
\eeqa
Moreover, as $t\to\infty$, uniformly in $j\in\Z$, 
\begin{equation}\label{gjj}
g_{11}(t,j)=\frac{2}{\sqrt{\pi t}}\e^{-\text\frac{j^2}{4t}}+\O(t^{-2})
\end{equation}
\begin{equation}\label{gtt}
g_{22}(t,j)=-\frac{1}{4t\sqrt{\pi t}}\Bigl(1-\frac{j^2}{2t}\Bigr)\,
\e^{-\text\frac{j^2}{4t}}+\O(t^{-2})
\end{equation}
\begin{equation}\label{gtj}
g_{12}(t,j)=-\frac{2j}{t\sqrt{\pi t}}\e^{-\text\frac{j^2}{4t}}+\O(t^{-2})
\end{equation}
\begin{equation}\label{integ}
\lim_{L\to\infty}\E (h^t_j-h^0_j)(h^t_0-h^0_0)
=\sqrt\frac{2t}{\pi}\,\Bigl[\,\e^{-\text\frac{j^2}{2t}}
+\frac{j}{\sqrt t}\int_{\text\frac{j}{\sqrt t}}^\infty du\, 
\e^{-u^2/2}\,\Bigr]\,+\O(\ln t)
\end{equation}
\end{prop}
\begin{rem}
(\ref{sym})(\ref{cause})(\ref{g1221}) follow respectively from
  space symmetry, causality, and the detailed balance condition.
More identities can be extracted from (\ref{tj}) and also from the
loop condition (the sum of gradients around a closed loop is
identically zero). In particular $g_{22}(t,0)=4g_{12}(t,1)\ \forall t$.
\end{rem} 
\begin{rem}
Proposition \ref{prop1} conveys information for
$|j|\ll\sqrt{t\ln t}$.  
\end{rem}
Proposition \ref{prop1} is proven in Section \ref{proof}. It is based on the
equivalence in law of the space-time field $h^{[0,t]}$ with the
infinite time limit of a space and time periodic random field, which
is naturally diagonalized by Fourier transform. This random field is
defined in Section \ref{periodic}. The Fourier transform
diagonalization is performed in Section \ref{fourier}. The proof of
equivalence is completed in Section \ref{onetime}. Generalization to arbitrary
dimension is outlined in Section \ref{higher}.
A numerical comparison to the harness process in continuous time is
given in Section \ref{rs}.

\section{Space-time periodic field}
\label{periodic}
For $T$ a positive even integer, the marginal space time field
\begin{equation}
h=\Bigl\{\,h^t_i\,:\ (t,i)\in\bigl(\,\{0,\dots T-1\}\times\Z/L\Z\,\bigr)
\cap\bigl\{t+i\ {\rm even}\bigr\}\,\Bigr\}
\end{equation}
is easily checked to be distributed according to the un-normalized measure
\begin{multline}
\mu^{TL}_{\rm free}(\d h)=\biggl(\,\prod_{t+i\ {\rm even}}
\e^{^{\text -\bigl(\,h^t_i-\12(h^{t-1}_{i-1}+h^{t-1}_{i+1})\,\bigr)^2}}
\d h^t_i\,\biggr).\cr
.\biggl(\,
\prod_{i\ {\rm even}}\e^{^{\text-\frac{1}{4}\bigl(\,h^0_i-h^0_{i+2}\,\bigr)^2}}
\d h^0_i\,\biggr)
\end{multline}
where ``free'' refers to the time $T-1$ final condition and the range of $t$ in
the product is $1\le t\le T-1$. The corresponding ``space-time Hamiltonian'' is
\begin{equation}\label{9}
\H^{TL}_{\rm free}=\sum_{\genfrac{}{}{0 pt}{}{i+t\,\rm even}{1\le t\le T-1}}
\Bigl(h^t_i-\12h^{t-1}_{i-1}-\12h^{t-1}_{i+1}\Bigr)^2
+\frac{1}{4}\sum_{i\ {\rm even}}\Bigl(h^0_i-h^0_{i+2}\Bigr)^2
\end{equation}
A good feature of $\mu^{TL}_{\rm free}$ is that 
its marginal at time $t$ is known exactly. However, in order to compute time 
correlations by Fourier transform, we are going to use periodic boundary 
conditions in the time variable also: let
\begin{equation}
h=\Bigl\{\,h^t_i\,:\ (t,i)\in\bigl(\Z/T\Z\times\Z/L\Z\bigr)
\cap\bigl\{t+i\ {\rm even}\bigr\}\,\Bigr\}
\end{equation}
be distributed according to the un-normalized measure
\begin{equation}
\mu^{TL}_{\rm per}(\d h)=\prod_{t+i\ {\rm even}}
\e^{^{\text -\bigl(\,h^t_i-\12(h^{t-1}_{i-1}+h^{t-1}_{i+1})\,\bigr)^2}}
\d h^t_i
\end{equation}
The corresponding ``space-time Hamiltonian'' is 
\begin{align}\label{12}
\H^{TL}_{\rm per}&=\sum_{i+t\ {\rm even}}\Bigl(h^t_i-\12h^{t-1}_{i-1}
-\12h^{t-1}_{i+1}\Bigr)^2\cr
=&\sum_{\genfrac{}{}{0 pt}{}{i+t\,\rm even}{1\le t\le T-1}}
\Bigl(h^t_i-\12h^{t-1}_{i-1}-\12h^{t-1}_{i+1}\Bigr)^2\
+\sum_{i\ {\rm even}}\Bigl(h^0_i-\12h^{T-1}_{i-1}
-\12h^{T-1}_{i+1}\Bigr)^2\cr
\end{align}

\bigskip
{\includegraphics[width=11cm]{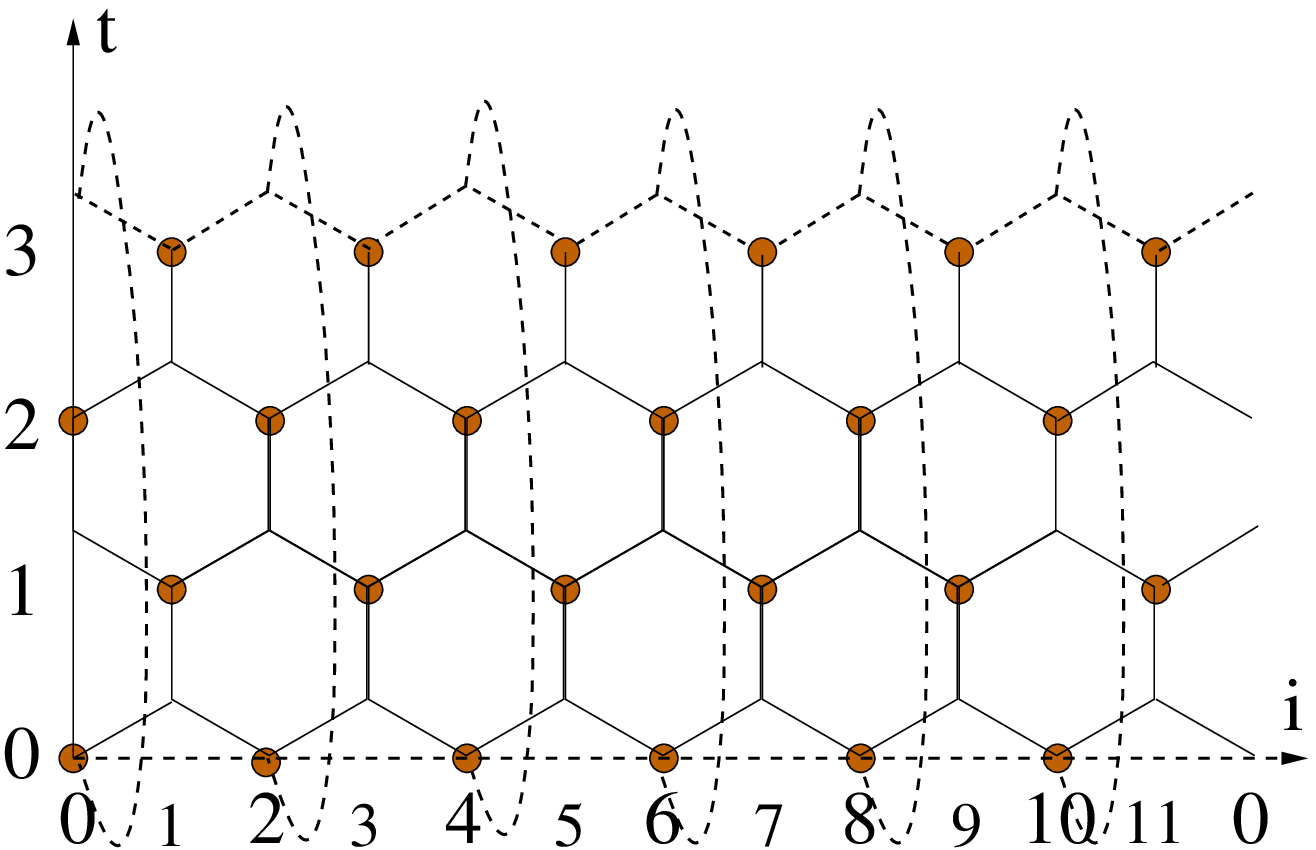}} 
\bigskip\noindent
\centerline{\small Fig.~1: Even space-time sub-lattice with
  $\H^{TL}_{\rm free}$ or $\H^{TL}_{\rm per}$, size $T=4$, $L=12$.}  
\bigskip

The last term in (\ref{12}) or (\ref{9}) is necessary in order to have only one
mode distributed according to the Lebesgue measure, for uniform global
translations of the system. Fig. 1 shows in solid line the interaction
terms common to  (\ref{12}) and (\ref{9}), and, in dashed line, the interaction
terms corresponding to the last term in (\ref{12}) or (\ref{9}).

\begin{prop}\label{prop2}
Let $L,T_1,T$ be positive even integers. Let 
$h^{TL}_{\rm per}$ and $h^{T_1L}_{\rm free}$ be random fields distributed 
according to $\mu^{TL}_{\rm per}$ and $\mu^{T_1L}_{\rm free}$ respectively. 
Then, as $T\to\infty$, the marginal $h^{TL}_{\rm per}\big|_{0\le t\le T_1-1}$ 
converges in distribution to $h^{T_1L}_{\rm free}$. In particular, the one-time
marginal  $h^0_{\rm per}$ converges in distribution to $h^0_{\rm free}$, 
distributed according to $\mu(\d h^0)$. And, extending the random fields
to the full lattice with $(h^t_{\rm per})_j=(h^{t+1}_{\rm per})_j$ for
$t+j$ odd, and similarly for $h_{\rm free}$, then $h^0_{\rm per}$
converges in distribution to $h^0_{\rm free}$, distributed according
to $\mu(\d h^0)$.
\end{prop}
\begin{proof}
Each random field has one real component distributed according to the
Lebesgue measure, the same for all random fields. We need only
consider the gradient fields.
The statements about the one-time marginals follow from the convergence of the 
one-time covariance matrix, e.g. $\E(h^0_i-h^0_{i-1})(h^0_j-h^0_{j-1})$
for $i,j=1,\dots,L-1$, which are linear combinations of $\E(h^0_j-h^0_0)^2$
for $j=1,\dots,L-1$. This is a computation, given in Section
\ref{onetime}. We thus have, for the gradient fields,
\begin{equation}
\P^{TL}_{\rm per}\bigl(\,\d h^0\,\bigr)\longrightarrow
\P^{L}_{\rm free}\bigl(\,\d h^0\,\bigr)\qquad{\rm as}\quad T\to\infty
\end{equation}
On the other hand, 
\begin{align}
\P^{TL}_{\rm per}\bigl(\,\d h^{[1,T_1-1]}\,|\,h^0\,\bigr)
&=\P^{L}_{\rm free}\bigl(\,\d h^{[1,T_1-1]}\,|\,h^0,\,h^T=h^0\bigr)\cr
&\longrightarrow\P^{L}_{\rm free}\bigl(\,\d h^{[1,T_1-1]}\,|\,h^0\bigr)
\qquad{\rm as}\quad T\to\infty
\end{align}
Then, for the gradient fields,
\begin{align}
\P^{TL}_{\rm per}\bigl(\,\d h^{[0,T_1-1]}\,\bigr)&=
\P^{TL}_{\rm per}\bigl(\,\d h^{[1,T_1-1]}\,|\,h^0\,\bigr)\,
\P^{TL}_{\rm per}\bigl(\,\d h^0\,\bigr)\cr
&\longrightarrow\P^{L}_{\rm free}\bigl(\,\d h^{[1,T_1-1]}\,|\,h^0\,\bigr)\,
\P^{L}_{\rm free}\bigl(\,\d h^0\,\bigr)\cr
&=\P^{L}_{\rm free}\bigl(\,\d h^{[0,T_1-1]}\,\bigr)
\end{align}
\end{proof}
\section{Fourier transform}
\label{fourier}
In order to compute the Fourier transform of the space-time periodic Gaussian 
field, it is convenient to set
\begin{equation}
h^{2t+1}_{2i}=0\quad{\rm and}\quad h^{2t+2}_{2i+1}=0\quad\forall\,t,i
\end{equation}
Then, for $(\nu,k)\in\bigl(\Z/T\Z\times\Z/L\Z\bigr)$,
\begin{equation}
\hh^\nu_k=\frac{1}{\sqrt{LT}}\sum_{j=0}^{L-1}\sum_{t=0}^{T-1}
\e^{^{\text 2\i\pi\frac{kj}{L}+2\i\pi\frac{\nu t}{T} }} h^t_j\ ,\qquad
\hh^{T-\nu}_{L-k}=\overline{\hh^\nu_k}\ ,\qquad
\hh^{\nu+T/2}_{k+L/2}={\hh^\nu_k}\qquad
\end{equation}
and
\begin{equation}\label{htj}
h^t_j=\frac{1}{\sqrt{LT}}\sum_{k=0}^{L-1}\sum_{\nu=0}^{T-1}
\e^{^{\text -2\i\pi\frac{kj}{L}-2\i\pi\frac{\nu t}{T} }}\hh^\nu_k
\end{equation}
and
\begin{align}\label{ganuk}
\H^{TL}_{\rm per}&=\sum_{j=0}^{L-1}\sum_{t=0}^{T-1}
\Bigl(h^t_j-\12h^{t-1}_{j-1}-\12h^{t-1}_{j+1}\Bigr)^2\cr
&=\frac{1}{LT}\sum_{j,t,\nu,k,\nu',k'}\hh^\nu_k\hh^{\nu'}_{k'}
\e^{^{\text -2\i\pi(k+k')\frac{j}{L}-2\i\pi(\nu+\nu')\frac{t}{T} }}\,.\cr
&\hskip3cm.\,\Bigl(1-\12\e^{^{\text 2\i\pi\frac{k}{L}+2\i\pi\frac{\nu}{T} }} 
-\12\e^{^{\text -2\i\pi\frac{k}{L}+2\i\pi\frac{\nu}{T} }}  \Bigr)\,.\cr
&\hskip3cm.\,\Bigl(1-\12\e^{^{\text 2\i\pi\frac{k'}{L}+2\i\pi\frac{\nu'}{T} }} 
-\12\e^{^{\text -2\i\pi\frac{k'}{L}+2\i\pi\frac{\nu'}{T} }}  \Bigr)\cr
&=\sum_{\nu,k}|\hh^\nu_k|^2\Bigl(1-\12\e^{^{\text 2\i\pi\frac{k}{L}
+2\i\pi\frac{\nu}{T}}} 
-\12\e^{^{\text -2\i\pi\frac{k}{L}+2\i\pi\frac{\nu}{T} }}  \Bigr)\,.\cr
&\hskip3cm.\Bigl(1-\12\e^{^{\text -2\i\pi\frac{k}{L}-2\i\pi\frac{\nu}{T} }} 
-\12\e^{^{\text 2\i\pi\frac{k}{L}-2\i\pi\frac{\nu}{T} }}  \Bigr)\cr
&=\sum_{\nu,k}|\hh^\nu_k|^2\Bigl(1-2\cos2\pi\frac{\nu}{T}\cos2\pi\frac{k}{L}
+\cos^22\pi\frac{k}{L}\Bigr)\cr
&=\sum_{\nu,k}|\hh^\nu_k|^2\,\gamma^\nu_k
\end{align}

The Fourier transform may be cast into an orthogonal transformation of the 
$LT/2$ random variables $h^t_i\,$.  Setting $\hh^\nu_k=a^\nu_k+ib^\nu_k$,
the new $LT/2$ random variables $a^\nu_k\,$'s and $b^\nu_k\,$'s are chosen as
follows: the orbit of any $(\nu,k)$ under possible combinations of
\begin{equation}
(\nu,\,k)\to(T-\nu,\,L-k)\qquad{\rm and}\qquad(\nu,\,k)\to(\nu+T/2,\,k+L/2)
\end{equation} 
has 4 elements, except for the 2-element orbits 
\begin{equation}
\{(0,0),\,(T/2,L/2)\}\,,\qquad\{(0,L/2),\,(T/2,0)\}
\end{equation}
\begin{equation}
\{(T/4,L/4),\,(3T/4,3L/4)\}\,,\qquad\{(T/4,3L/4),\,(3T/4,L/4)\}
\end{equation} 
corresponding to real $\hh^\nu_k$'s. Choosing one element per orbit, we get
\begin{equation}\label{20}
\begin{split}
h^t_j&=
\frac{2}{\sqrt{LT}}\Bigl[a^0_0+(-)^ja^0_{L/2}+(-)^\frac{t+j}{2}a^{T/4}_{L/4}
+(-)^\frac{t-j}{2}a^{T/4}_{3L/4}\Bigr]\\
&+\frac{4}{\sqrt{LT}}\sum_{k=1}^{L/2-1}\Bigl[\cos2\pi\frac{kj}{L}\,a^0_k
+\sin2\pi\frac{kj}{L}\,b^0_k\Bigr]\\
&+\frac{4}{\sqrt{LT}}\sum_{k=L/4+1}^{3L/4-1}
\Bigl[\cos\bigl(2\pi\frac{kj}{L}+\frac{\pi t}{2}\bigr)\,a^{T/4}_k
+\sin\bigl(2\pi\frac{kj}{L}+\frac{\pi t}{2}\bigr)\,b^{T/4}_k\Bigr]\\
&+\frac{4}{\sqrt{LT}}\sum_{\nu=1}^{T/4-1}\sum_{k=0}^{L-1}
\Bigl[\cos\bigl(2\pi\frac{kj}{L}+2\pi\frac{\nu t}{T}\bigr)\,a^\nu_k
+\sin\bigl(2\pi\frac{kj}{L}+2\pi\frac{\nu t}{T}\bigr)\,b^\nu_k\Bigr]
\end{split}
\end{equation} 
The jacobian of the transformation from $h^t_i\,$'s to $a^\nu_k\,$'s and 
$b^\nu_k\,$'s is actually $2^{TL/4}$. The new measure is
\begin{equation}
\exp(-\H^{TL}_{\rm per})\prod_{\nu,k}\d a^\nu_k\prod_{\nu,k}\d b^\nu_k\,/\,
{\rm norm.}
\end{equation}
where the index sets for the $a^\nu_k\,$'s and $b^\nu_k\,$'s are as
in the formula (\ref{20}) for $h^t_j$, with a total of $LT/2$, and
\begin{equation}
\H^{TL}_{\rm per}=2\sum_{2-{\rm orbits}}\gamma^\nu_k(a^\nu_k)^2
+4\sum_{4-{\rm orbits}}\gamma^\nu_k\Bigl[(a^\nu_k)^2+(b^\nu_k)^2\Bigr]
\end{equation}
with $\gamma^\nu_k$ defined in (\ref{ganuk}), so that
$\E|\hh^\nu_k|^2=1/(4\gamma^\nu_k)$ for all $(\nu,k)\ne(0,0)$ or $(T/2,L/2)$.
We have one zero mode $a^0_0=a^{T/2}_{L/2}$, distributed with Lebesgue measure,
and soft modes, Gaussians of large variance, around the zero mode:
\begin{lem}\label{hnuk}
$\hh^0_0=\hh^{T/2}_{L/2}$ is distributed according to the Lebesgue
measure. All other $\hh^\nu_k$ are independent centred real or complex
Gaussian variables with
\begin{equation}
\E|\hh^\nu_k|^2=\frac{1}{4}\,
\frac{1}{1-2\cos2\pi\frac{\nu}{T}\cos2\pi\frac{k}{L}+\cos^22\pi\frac{k}{L}}
\end{equation}
$\E\hh^\nu_k\hh^{\nu'}_{k'}$ is non zero only if $(\nu,k)$ and $(\nu',k')$
belong to the same orbit, with $\hh^\nu_k$ and $\hh^{\nu'}_{k'}$ complex 
conjugate:
$(\nu',k')=(T-\nu,L-k)$ or $(\nu',k')=(T/2-\nu,L/2-k)$. 
\end{lem}
\section{Equal time covariance}
\label{onetime}
The equilibrium measure (\ref{eqm}) can also be diagonalized by Fourier
transform, which yields
\begin{equation}
\E(h^0_j-h^0_0)^2=\frac{1}{L}\sum_{k\ne0}
\frac{1-\cos\frac{2\pi kj}{L}}{1-\cos\frac{2\pi k}{L}}
\end{equation}
where $k\in\Z/L\Z$. A change of summation index $k\to\frac{L}{2}-k$
leads to equivalent formulas, differing slightly according to the
parity of $j$. Averaging the two formulas yields
\begin{align}\label{25}
\E(h^0_j-h^0_0)^2=&
\frac{1}{L}\sum_{k\ne0,L/2}\frac{1-\cos\frac{2\pi kj}{L}}{1-\cos^2
\frac{2\pi k}{L}}\qquad{\rm if}\quad j\ {\rm even}\hskip2cm\cr
=&\frac{1}{L}+\frac{1}{L}\sum_{k\ne0,L/2}\frac{1-\cos\frac{2\pi kj}{L}
\cos\frac{2\pi k}{L}}{1-\cos^2\frac{2\pi k}{L}}\qquad{\rm if}\quad j\ 
{\rm odd}\cr
\end{align}
Here of course $L=\infty$ would be simpler, with $\E(h^0_j-h^0_0)^2=j$, 
but our aim is to complete the proof of Proposition 2, where $L$
is finite.
Let us now compute the analogue for the space and time periodic field:
\begin{equation}
\E_{\rm per}(h^0_j-h^0_0)^2=\frac{1}{LT}\sum_{\nu,k,\nu',k'}
\bigl(\e^{^{\text -2\i\pi\frac{kj}{L} }}-1\bigr)
\bigl(\e^{^{\text -2\i\pi \frac{k'j}{L} }}-1\bigr)\E\hh^\nu_k\hh^{\nu'}_{k'}
\end{equation}
Therefore, with $j$ even, using Lemma \ref{hnuk},
\begin{align}\label{27}
\E_{\rm per}(h^0_j-h^0_0)^2&=\frac{2}{LT}\sum_{\nu,k}
\bigl(\e^{^{\text -2\i\pi\frac{kj}{L} }}-1\bigr)
\bigl(\e^{^{\text 2\i\pi\frac{kj}{L} }}-1\bigr)\E |\hh^\nu_k|^2\cr
&=\frac{1}{LT}\sum_{k\ne0,L/2}\sum_\nu\frac{1-\cos2\pi\frac{kj}{L}}
{1-2\cos2\pi\frac{\nu}{T}\cos2\pi\frac{k}{L}+\cos^22\pi\frac{k}{L}}\cr
\longrightarrow\
\frac{1}{L}\sum_{k\ne0,L/2}\frac{1}{2\pi}&\int_0^{2\pi}\d\om\, 
\frac{1-\cos2\pi\frac{kj}{L}} 
{1-2\cos\om\cos2\pi\frac{k}{L}+\cos^22\pi\frac{k}{L}}
\qquad{\rm as}\qquad T\to\infty\cr
&=\frac{1}{L}\sum_{k\ne0,L/2}\frac{1-\cos2\pi\frac{kj}{L}}
{1-\cos^22\pi\frac{k}{L}}
\end{align}
where the last step comes from the identity  \cite{[GR]}, p 366,
\begin{equation}\label{GR}
\frac{1}{2\pi}\int_0^{2\pi}\d\om\,\frac{\cos n\om}{1-2a\cos\om+a^2}
=\frac{a^n}{1-a^2}
\end{equation}
The covariance (\ref{27}) is indeed the same as the covariance
(\ref{25}). Similarly, for $j$ odd,
\begin{align}
\E_{\rm per}(h^1_j-h^0_0)^2&=\frac{1}{LT}\sum_{\nu k\nu'k'}
\bigl(\e^{^{\text-\frac{2\i\pi kj}{L}-\frac{2\i\pi\nu}{T} }}-1\bigr)
\bigl(\e^{^{\text-\frac{2\i\pi k'j}{L}-\frac{2\i\pi\nu'}{T} }}-1\bigr)
\E\hh^\nu_k\hh^{\nu'}_{k'}\cr
&=\frac{2}{LT}\sum_{\nu,k}
\bigl(\e^{^{\text -\frac{2\i\pi kj}{L}-\frac{2\i\pi\nu}{T} }}-1\bigr)
\bigl(\e^{^{\text\frac{2\i\pi kj}{L}+\frac{2\i\pi\nu}{T}}}-1\bigr)
\E|\hh^\nu_k|^2\cr
&=\frac{1}{LT}\sum_{(k,\nu)\ne(0,0),(\frac{L}{2},\frac{T}{2})}
\frac{1-\cos(2\pi\frac{kj}{L}+2\pi\frac{\nu}{T})}
{1-2\cos2\pi\frac{\nu}{T}\cos2\pi\frac{k}{L}+\cos^22\pi\frac{k}{L}}\cr
&=\frac{1}{L}\frac{T-1}{ T}+\frac{1}{LT}\sum_{k\ne0,\frac{L}{2}}
\sum_{\nu\ne0,\frac{T}{2}}\frac{1-\cos(2\pi\frac{kj}{L}+2\pi\frac{\nu}{T})}
{1-2\cos2\pi\frac{\nu}{T}\cos2\pi\frac{k}{L}+\cos^22\pi\frac{k}{L}}\cr
\longrightarrow\ \frac{1}{L}+\frac{1}{L}&\sum_{k\ne0,L/2}\frac{1}{2\pi}
\int_0^{2\pi}\d\om\,\frac{1-\cos(2\pi\frac{kj}{L}+\om)}
{1-2\cos\om\cos2\pi\frac{k}{L}+\cos^22\pi\frac{k}{L}}
\quad{\rm as}\quad T\to\infty\cr
&=\frac{1}{L}+\frac{1}{L}\sum_{k\ne0,L/2}\frac{1}{2\pi}
\int_0^{2\pi}\d\om\,\frac{1-\cos(2\pi\frac{kj}{L})\cos\om}
{1-2\cos\om\cos2\pi\frac{k}{L}+\cos^22\pi\frac{k}{L}}\cr 
&=\frac{1}{L}+\frac{1}{L}\sum_{k\ne0,L/2}\frac
{1-\cos2\pi\frac{kj}{L}\cos2\pi\frac{k}{L}}{1-\cos^22\pi\frac{k}{L}}
\end{align}
which again is the same as (\ref{25}).
This completes the proof of Proposition 2.

\section{Proof of Proposition \ref{prop1}}\label{proof}
\medskip\noindent{\sl Proof of (\ref{tj})(\ref{gjj}):} using (\ref{htj})
and Lemma \ref{hnuk}, with $t$ even,
\begin{align}
\E_{\rm per}(h^t_{j+2}-&h^t_j)(h^0_2-h^0_0)=\cr
=\frac{2}{LT}\sum_{\nu,k}
\bigl(\e&^{^{\text -2\i\pi\frac{k(j+2)}{L}-2\i\pi\frac{\nu t}{T}}}
-\e^{^{\text -2\i\pi\frac{kj}{L} -2\i\pi\frac{\nu t}{T}}}\bigr)
\bigl(\e^{^{\text 4\i\pi\frac{k}{L} }}-1\bigr)\E|\hh^\nu_k|^2\cr
&=\frac{1}{LT}\sum_{\nu}\sum_{k\ne0,L/2}\frac{\cos2\pi\frac{kj}{L}
\cos2\pi\frac{\nu t}{T}\,\bigl(1-\cos4\pi\frac{k}{L}\bigr)}
{1-2\cos2\pi\frac{\nu}{T}\cos2\pi\frac{k}{L}+\cos^22\pi\frac{k}{L}}\hskip2cm\cr
\longrightarrow\frac{1}{ 4\pi^2} \int_0^{2\pi}&\d\om\int_0^{2\pi}\d\phi\, 
\frac{\cos\phi j\cos\om t\,(1-\cos2\phi)}
{1-2\cos\om\cos\phi+\cos^2\phi}\qquad {\rm as}\ L,T\to\infty
\end{align}
so that
\begin{align}\label{g11phi}
\lim_{L\to\infty}\E (h^{t}_{j+2}-h^{t}_j)(h^0_2-h^0_0)
&=\frac{1}{4\pi^2}\int_0^{2\pi}\d\om\int_0^{2\pi}\d\phi\, 
\frac{\cos\phi j\cos\om t\,(1-\cos2\phi)}{1-2\cos\om\cos\phi+\cos^2\phi}\cr
&=\frac{1}{\pi}\int_0^{2\pi}\d\phi\,\cos\phi j\,\bigl(\cos\phi\bigr)^t\cr
&=\frac{2^{-t+1}t!}{\frac{t-j}{2}!\frac{t+j}{2}!}
\end{align}
where we used (\ref{GR}) from the first to the second line, and the
last line assumes $j\le t$, otherwise the result is zero. Changing $t$
into $2t$ yields (\ref{tj}). Stirling's formula then leads to (\ref{gjj}).

\bigskip\noindent{\sl Proof of (\ref{g2211})(\ref{gtt}):}
\smallskip
\begin{align}
\E_{\rm per}(h^{t+2}_j-&h^t_j)(h^2_0-h^0_0)=\cr
=\frac{2}{LT}\sum_{\nu,k}
\bigl(\e&^{^{\text -2\i\pi\frac{kj}{L}-2\i\pi\frac{\nu (t+2)}{T}}}
-\e^{^{\text -2\i\pi\frac{kj}{L} -2\i\pi\frac{\nu t}{T}}}\bigr)
\bigl(\e^{^{\text 4\i\pi\frac{\nu}{T} }}-1\bigr)\E|\hh^\nu_k|^2\cr
&=\frac{1}{LT}\sum_{\nu\ne0,T/2}\sum_k\frac{\cos2\pi\frac{kj}{L}
\cos2\pi\frac{\nu t}{T}\,\bigl(1-\cos4\pi\frac{\nu}{T}\bigr)}
{1-2\cos2\pi\frac{\nu}{T}\cos2\pi\frac{k}{L}+\cos^22\pi\frac{k}{L}}\hskip2cm\cr
\longrightarrow\frac{1}{ 4\pi^2} \int_0^{2\pi}&\d\om\int_0^{2\pi}\d\phi\, 
\frac{\cos\phi j\cos\om t\,(1-\cos2\om)}
{1-2\cos\om\cos\phi+\cos^2\phi}\qquad {\rm as}\ L,T\to\infty
\end{align}
so that
\begin{align}
\lim_{L\to\infty}\E (h^{t+2}_j-h^t_j)(h^2_0-h^0_0)
&=\frac{1}{4\pi^2}\int_0^{2\pi}\d\om\int_0^{2\pi}\d\phi\, 
\frac{\cos\phi j\cos\om t\,(1-\cos2\om)}{1-2\cos\om\cos\phi+\cos^2\phi}\cr
&=-\frac{1}{4\pi}\int_0^{2\pi}\d\phi\,\cos\phi j\Bigl[
(\cos\phi)^{(t-2)}-(\cos\phi)^t\Bigr]
\end{align}
where (\ref{GR}) was used once more.
Comparing with the second line of (\ref{g11phi}) gives (\ref{g2211}),
which combined with (\ref{gjj}) gives (\ref{gtt}).

\bigskip\noindent{\sl Proof of (\ref{g1211}):}
\smallskip
\begin{align}\label{g12per}
\E_{\rm per}(h^t_{j+1}-&h^t_{j-1})(h^2_0-h^0_0)=\cr
=\frac{2}{LT}\sum_{\nu,k}
&\e^{^{\text -2\i\pi\frac{\nu t}{T}-2\i\pi\frac{kj}{L} }}
\bigl(\e^{^{\text 4\i\pi\frac{\nu}{T} }}-1\bigr)
\bigl(\e^{^{\text -2\i\pi\frac{k}{L}}}-\e^{^{\text -2\i\pi\frac{k}{L}}}\bigr)
\E|\hh^\nu_k|^2\cr
=-\frac{4}{LT}\sum_{\nu,k}&\e^{^{\text -2\i\pi\frac{\nu t}{T} }}
\bigl(\e^{^{\text 4\i\pi\frac{\nu}{T} }}-1\bigr)
\sin 2\pi{\text\frac{kj}{L}}\,\sin 2\pi{\text\frac{k}{L}}\,\E|\hh^\nu_k|^2
\end{align}
where the space symmetry $j\to-j$ was used to get the last line. In
order to use time reversal symmetry, coming with the detailed balance equation,
let us compute similarly
\begin{align}\label{g21per}
\E_{\rm per}(h^0_{j+1}-&h^0_{j-1})(h^t_0-h^{t-2}_0)=\cr
=&\frac{4}{LT}\sum_{\nu,k}\e^{^{\text 2\i\pi\frac{\nu t}{T} }}
\bigl(\e^{^{\text -4\i\pi\frac{\nu}{T} }}-1\bigr)
\sin 2\pi{\text\frac{kj}{L}}\,\sin 2\pi{\text\frac{k}{L}}\,\E|\hh^\nu_k|^2
\end{align}
Time reversal shows that (\ref{g21per}) is the opposite of (\ref{g12per}),
which therefore equals
\begin{align}
\E_{\rm per}(h^t_{j+1}-&h^t_{j-1})(h^2_0-h^0_0)=\cr
=\frac{4}{LT}\sum_{\nu,k}&\bigl(\cos{\text 2\pi\frac{\nu t}{T} }
-\cos{\text 2\pi\frac{\nu(t-2)}{T} }\bigr)
\sin 2\pi{\text\frac{kj}{L}}\,\sin 2\pi{\text\frac{k}{L}}\,\E|\hh^\nu_k|^2\cr
\longrightarrow\frac{1}{\pi^2} \int_0^{2\pi}&\d\om\int_0^{2\pi}\d\phi\, 
\frac{\bigl(\cos\om t-\cos\om(t-2)\bigr)\sin\phi j\sin\phi}
{1-2\cos\om\cos\phi+\cos^2\phi}\qquad {\rm as}\ L,T\to\infty\cr
\end{align}
so that, using (\ref{GR}) once more,
\beq
\lim_{L\to\infty}\E (h^{t}_j-h^{t-2}_j)(h^2_0-h^0_0)
=\frac{1}{\pi}\int_0^{2\pi}\d\phi\,(\cos\phi)^{(t-2)}
\bigl(\cos\phi(j+1)-\cos\phi(j-1)\bigr)\eeq
Comparing with the second line of (\ref{g11phi}) gives (\ref{g1211}),
which combined with (\ref{gjj}) gives (\ref{gtj}).

\bigskip\noindent{\sl Second proof of (\ref{gtt}):}
\smallskip
We give here a second proof of (\ref{gtt}), not relying upon (\ref{g2211}).
Let
\begin{equation}
g(\om,j)=(1-\cos2\om)\int_0^{2\pi} \frac
{\d\phi\,\cos\phi j}{ 1-2\cos\om\cos\phi+\cos^2\phi}=g(-\om,j)=g(\pi-\om,j)
\end{equation}
so that
\begin{align}
f(t,j)&\equiv\int_0^{2\pi}\d\om\int_0^{2\pi}\d\phi\, 
\frac{\cos\phi j\cos\om t\,(1-\cos2\om)}{1-2\cos\om\cos\phi+\cos^2\phi}\cr
&=\int_0^{2\pi}\d\om\,\e^{\i\om t}g(\om,j)\cr
&=-\frac{1}{\i t}\int_0^{2\pi}\d\om\,\e^{\i\om t}g'(\om,j)\cr
&=-\frac{2}{\i t}\,\int_0^{\pi/2}\d\om\, 
\e^{\i\om t}g'(\om,j)+{\rm c.c.}
\end{align}
In order to perform the integral over $\phi$ and estimate $g'(\om,j)$
near $\om=0$, we decompose
\begin{equation}
\frac{1}{1-2\cos\om\cos\phi+\cos^2\phi}
=\frac{1}{1-\e^{-2\i\om}}\,\frac{1}{1-\e^{\i\om}\cos\phi}+{\rm c.c.}
\end{equation}
and use the residue theorem: 
\begin{align}
\int_{-\pi}^\pi\frac{\d\phi\,\e^{\i\phi j}}{1-\e^{\i\om}\cos\phi}
&=\O(j^{-1})+2\e^{-\i\om}\int_{-\infty}^\infty
\frac{\d\phi\,\e^{\i\phi j}}{\phi^2-(2-2\e^{-\i\om})}\cr
&=\O(j^{-1})+2\e^{-\i\om}\frac{2\i\pi}{2\sqrt{2-2\e^{-\i\om}}}\,
\e^{\i\sqrt{2-2\e^{-\i\om}}\,j}\cr
&=\O(j^{-1})+\O(\om^{1/2})+\sqrt{2}\pi\e^{\i\pi/4}\om^{-1/2}\,
\e^{-\sqrt{2}\,\text \e^{-\i\pi/4}\om^{1/2}j}\cr
&=\int_{-\pi}^\pi\frac{\d\phi\,\cos\phi j}{1-\e^{\i\om}\cos\phi}
\end{align}
because the integral with $\e^{-\i\phi j}$, with the contour closed in
the lower complex half-plane, gives an equal contribution. 
The $\O(j^{-1})$ is a regular function of $\om$, bounded by
const$.\,j^{-1}$ as $j\to\infty$. Therefore
\begin{equation}
g(\om,j)=\O(\om j^{-1})+\O(\om^{3/2})+\sqrt{2}\pi\e^{-\i\pi/4}\om^{1/2}\,
\e^{-\sqrt{2}\,\text \e^{-\i\pi/4}\om^{1/2}j}+{\rm c.c.}
\end{equation}
and
\begin{multline}\label{gp}
g'(\om,j)=\O(j^{-1})+\O(\om^{1/2})
+\i\pi j\,\e^{-\sqrt{2}\,\text\e^{-\i\pi/4}\om^{1/2}j}\cr  
+\frac{\pi}{\sqrt{2}}\e^{-\i\pi/4}
\om^{-1/2}\,\e^{-\sqrt{2}\,\text\e^{-\i\pi/4}\om^{1/2}j}+{\rm c.c.}
\end{multline}
In (\ref{gp}), the phase factor oscillating with $\om$ is
$\exp(\i\om^{1/2}j)$ in both terms shown in full, whereas the
c.c. counterparts will have  
$\exp(-\i\om^{1/2}j)$. When integrating against $\exp(\i\om t)$ the
leading contributions will therefore come from the c.c. counterparts,
which will give a stationary phase region. This agrees with the
result of the calculation below. 
\begin{align}
\int_0^{\pi/2}\d\om\,
\,\e^{\i\om t-\sqrt{2}\,\text \e^{\i\pi/4}\om^{1/2}j}
&=\O(t^{-1}e^{-\sqrt{\frac{\pi}{2}}j})+\int_0^\infty\d\om\,
\e^{\i\om t-\sqrt{2}\,\text\e^{\i\pi/4}\om^{1/2}j}\cr
&=\O(t^{-1}e^{-\sqrt{\frac{\pi}{2}}j})+2\i\int_0^{\e^{-\i\pi/4}\infty}x\d
x\,\e^{-tx^2-\i\sqrt{2}jx}\cr 
&=\O(t^{-1}e^{-\sqrt{\frac{\pi}{2}}j})+2\i\int_0^{\infty}x\d
x\,\e^{-tx^2-\i\sqrt{2}jx} 
\end{align}
\begin{align}
\int_0^{\pi/2}\d\om\,\om^{-1/2}\,
\e^{\i\om t-\sqrt{2}\,\text\e^{\i\pi/4}\om^{1/2}j}
&=\O(t^{-1})+\int_0^\infty\d\om\,\om^{-1/2}\,
\e^{\i\om t-\sqrt{2}\,\text\e^{\i\pi/4}\om^{1/2}j}\cr
&=\O(t^{-1})+2\e^{\i\pi/4}
\int_0^{\e^{-\i\pi/4}\infty}\d x\,\e^{-tx^2-\i\sqrt{2}jx}\cr
&=\O(t^{-1})+2\e^{\i\pi/4}\int_0^{\infty}\d x\,\e^{-tx^2-\i\sqrt{2}jx}
\end{align}
so that
\begin{multline}
f(t,j)=\O(t^{-2})+\i\frac{4\pi j}{t}
\int_0^\infty x\d x\,\e^{-tx^2-\i\sqrt{2}jx} 
-\frac{4\pi}{t\sqrt2}\int_0^{\infty}\d x\,\e^{-tx^2-\i\sqrt{2}jx}+{\rm c.c.}
\end{multline}
Then
\begin{equation}
\i\frac{4\pi j}{t}\int_0^\infty x\d x\,\e^{-tx^2-\i\sqrt{2}jx}+{\rm  c.c.}
=\i\frac{4\pi j}{t}\int_{-\infty}^\infty x\d x\,\e^{-tx^2-\i\sqrt{2}jx}
=\frac{4\pi^{3/2}j^2}{\sqrt2\,t^{5/2}}\,\e^{-\text\frac{j^2}{2t}}
\end{equation}
and
\begin{equation}
-\frac{4\pi}{t\sqrt2}\int_0^{\infty}\d x\,\e^{-tx^2-\i\sqrt{2}jx}+{\rm c.c.}
=-\frac{4\pi}{t\sqrt2}\int_{-\infty}^{\infty}\d x\,\e^{-tx^2-\i\sqrt{2}jx}
=-\frac{4\pi^{3/2}}{\sqrt2\,t^{3/2}}\,\e^{-\text\frac{j^2}{2t}}
\end{equation}
Altogether
\begin{equation}
f(t,j)=\O(t^{-2})-\frac{4\pi^2}{\sqrt{2\pi}\,t^{3/2}}
\Bigl(1-\frac{j^2}{t}\Bigr)\,\e^{-\text\frac{j^2}{2t}}
\end{equation}
which completes the second proof of (\ref{gtt}).

\bigskip\noindent{\sl Proof of (\ref{integ}):}
\smallskip
\begin{align}
\E_{\rm per}(h^t_j-h^0_j)(h^t_0-&h^0_0)=\frac{2}{LT}\sum_{\nu,k}
\bigl(\e^{^{\text -2\i\pi\frac{kj}{L}-2\i\pi\frac{\nu t}{T} }}
-\e^{^{\text -2\i\pi\frac{kj}{L} }}\bigr)\,.\cr
&\hskip5cm.\,
\bigl(\e^{^{\text 2\i\pi\frac{\nu t}{T} }}-1\bigr)\E |\hh^\nu_k|^2\cr
&=\frac{1}{LT}\sum_{\nu,k}
\frac{\bigl(\cos2\pi\frac{kj}{L}\bigr)\bigl(1-\cos2\pi\frac{\nu t}{T}\bigr)}
{ 1-2\cos2\pi\frac{\nu}{T}\cos2\pi\frac{k}{L}+\cos^22\pi\frac{k}{L}
}\hskip2cm\cr
\longrightarrow\frac{1}{4\pi^2} \int_0^{2\pi}&\d\om\int_0^{2\pi}\d\phi\, 
\frac{\bigl(1-\cos\om t\bigr)\cos\phi j}
{1-2\cos\om\cos\phi+\cos^2\phi}\quad{\rm as}\quad L,T\to\infty\cr
=\frac{1}{\pi^2}\int_0^{\pi/2}\d\om&\int_0^{2\pi}\d\phi\, 
\frac{\bigl(1-\cos\om t\bigr)\cos\phi j}
{1-2\cos\om\cos\phi+\cos^2\phi}\qquad (\,t,j\ {\rm  even}\,)
\end{align}
The integral over $\phi$ is done like in the proof of (\ref{gtt}). Then,
using
\[
1-\cos\om t=\om\int_0^t\d s \sin\om s\,,
\]
\begin{align}
\E (h^t_j-h^0_j)(h^t_0-h^0_0)=\O(\ln t)+\hskip6cm\cr
+\frac{1}{2^{1/2}\pi}\int_0^tds
\int_0^{\pi/2}\d\om\ \e^{-\i\pi/4}\om^{-1/2}\sin\om s\,
\,\e^{-\sqrt{2}\,\text \e^{-\i\pi/4}\om^{1/2}j}+{\rm c.c.}
\end{align}
Then, setting $\om=\i x^2$ and using Cauchy's theorem,
\begin{align}
\frac{1}{2\i}\int_0^{\frac{\pi}{2}}\d\om\ \e^{-\i\frac{\pi}{4}}\om^{-\frac{1}{2}}
\,\e^{\text \i\om s-{\scr\sqrt{2}}\,\text \e^{-\i\frac{\pi}{4}}\om^{\frac{1}{2}}j}
&=\O(s^{-1})
+\frac{1}{\i}\int_0^{\e^{-\i\frac{\pi}{4}}\infty}\d x\,\e^{-sx^2-\sqrt{2}jx}\cr
&=\O(s^{-1})+\frac{1}{\i}\int_0^{\infty}\d x\,\e^{-sx^2-\sqrt{2}jx}
\end{align}
whose main part will cancel out with its complex conjugate.
Similarly, setting $\om=-\i x^2$ and using Cauchy's theorem,
\begin{align}
-\frac{1}{2\i}\int_0^{\frac{\pi}{2}}\d\om\ \e^{-\i\frac{\pi}{4}}\om^{-\frac{1}{2}}
\,\e^{\text -\i\om s-{\scr\sqrt{2}}\,\text \e^{-\i\frac{\pi}{4}}\om^{\frac{1}{2}}j}
&=\O(s^{-1})+\int_0^{\e^{\i\frac{\pi}{4}}\infty}\d x\,\e^{-sx^2+\i\sqrt{2}jx}\cr
&=\O(s^{-1})+\int_0^{\infty}\d x\,\e^{-sx^2+\i\sqrt{2}jx}
\end{align}
and
\begin{equation}
\int_0^{\infty}\d x\,\e^{-sx^2+\i\sqrt{2}jx}+{\rm c.c.}
=\int_{-\infty}^{\infty}\d x\,\e^{-sx^2+\i\sqrt{2}jx}
=\sqrt\frac{\pi}{ s}\,\e^{-\text\frac{j^2}{2s}}
\end{equation}
Altogether
\begin{align}
\E (h^t_j-h^0_j)(h^t_0-h^0_0)&=
\frac{1}{\sqrt{2\pi}}\int_0^t\frac{ds}{\sqrt s}\,
\e^{-\text\frac{j^2}{2s}}\,+\O(\ln t)\cr
&=\sqrt\frac{2t}{\pi}\int_1^\infty\frac{du}{u^2}\,
\e^{-\text\frac{j^2}{2t}u^2}\,+\O(\ln t)\cr
&=\sqrt\frac{2t}{\pi}\,\Bigl[\,\e^{-\text\frac{j^2}{2t}}
+\frac{j}{\sqrt t}\int_{\text\frac{j}{\sqrt t}}^\infty du\, 
\e^{-u^2/2}\,\Bigr]\,+\O(\ln t)
\end{align}
which completes the proof of Proposition \ref{prop1}.
\section{Higher dimension}
\label{higher}
In arbitrary $d\ge1$, $h^0=\{h^0_i:i\in\bigl(\Z/L\Z\bigr)^d\,\}$, the initial
measure 
\begin{equation}
\mu(\d h^0)=\prod_{|i-j|=1}\e^{^{\text
    -\frac{1}{2}\bigl(\,h^0_i-h^0_j\bigr)^2}} \prod_i\d h^0_i
\end{equation}
is invariant under the dynamics defined by
\begin{multline}
\P\bigl(\,\d h^t\,|\,h^{t-1}\,\bigr)=\prod_{|i|+t\ {\rm even}}
\e^{^{\text -d\bigl(\,h^t_i-\frac{1}{2d}\sum_jh^{t-1}_j\,\bigr)^2}}\,.\cr
.\,\prod_{|i|+t\ {\rm odd}}\del(h^t_i-h^{t-1}_i)\prod_i\d h^t_i
\bigg/{\rm norm.}
\end{multline}
where $\sum_j$ runs over the $2d$ neighbors of $i$.
The space-time Hamiltonian on the $d+1$-dimensional torus 
$\bigl(\Z/T\Z\bigr)\times\bigl(\Z/L\Z\bigr)^d$ is
\begin{equation}
\H_{\rm per}=d\sum_{i,t}\Bigl(\,h^t_i-\frac{1}{2d}
\sum_{j:|i-j|=1}h^{t-1}_j\,\Bigr)^2
\end{equation}             
The Fourier transform is defined as in $d=1$, using ${\bf k}=(k_1\dots k_d)$.
Then
\begin{align}
\H_{\rm per}&=\sum_{\nu,{\bf k}}|\hat h^\nu_{\bf k}|^2
\biggl[1-2\Bigl(\cos2\pi\frac{\nu}{T}\Bigr)
\frac{1}{d}\sum_{n=1}^d\cos2\pi\frac{k_n}{L}
+\frac{1}{d^2}\Bigl(\sum_{n=1}^d\cos2\pi\frac{k_n}{L}\Bigr)^2\biggr]\cr
&=\sum_{\nu,{\bf k}}|\hh^\nu_{\bf k}|^2\,\gamma^\nu_{\bf k}\cr
&\sim\frac{1}{2}\sum_{\nu,{\bf k}}|\hat h^\nu_{\bf k}|^2
\biggl[
\frac{(2\pi\nu)^2}{T^2}\Bigl(1-\frac{1}{2d}\frac{(2\pi{\bf k})^2}{L^2}\Bigr)
+\frac{1}{d^2}\Bigl(\frac{(2\pi{\bf k})^2}{2L^2}\Bigr)^2\biggr]
\end{align}
and the autocorrelation can be computed, yielding the expected $\ln(t)$ in 
$d=2$ or $\O(1)$ in $d\ge3$.
\section{Harness process in continuous time}
\label{rs}
The harness process in continuous time 
can be constructed as the $L\to\infty$ limit of the harness process
with random sequential update, defined like
the sub-lattice parallel dynamics but with (\ref{slpd}) replaced by
\begin{multline}\label{rsd}
\P\bigl(\,\d h^\tau\,|\,h^{\tau-1}\,\bigr)=\sum_{j=0}^{L-1}
\e^{^{\text-\bigl(\,h^\tau_j-\12(h^{\tau-1}_{j-1}+h^{\tau-1}_{j+1})
\,\bigr)^2}}\,.\\
.\prod_{i\neq j}\del(h^\tau_i-h^{\tau-1}_i)\prod_i\d h^\tau_i\bigg/{\rm norm.}
\end{multline}
The time $t$ for the Poisson clocks of rate one in the harness process
in continuous time is related to the microscopic time $\tau$ through
$t=\tau/L$. The measure (\ref{eqm}) is also invariant under the
dynamics (\ref{rsd}), and we still take it as initial condition. 

Here we give numerical results for this model, indicating that the
asymptotic forms (\ref{gjj})(\ref{gtt})(\ref{gtj}) in Proposition
\ref{prop1} may still be valid, with some rescaling.
The initial condition is drawn using the Fourier modes, which are
independent Gaussian variables. We then run the dynamics for a time
$t_1+t$ and measure the correlations: 
\beq
g_{11}^{L,t_1}(t,j)=\frac{1}{L}\sum_{i=0}^{L-1}\frac{1}{t_1}
\sum_{t'=0}^{t_1-1}\bigl(h_{i+2}^{Lt'}-h_i^{Lt'}\bigr)
\bigl(h_{i+j+2}^{L(t'+t)}-h_{i+j}^{L(t'+t)}\bigr)
\eeq
\beq
g_{22}^{L,t_1}(t,j)=\frac{1}{L}\sum_{i=1}^L\frac{1}{t_1}
\sum_{t'=0}^{t_1-1}\bigl(h_{i}^{L(t'+1)}-h_i^{Lt'}\bigr)
\bigl(h_{i+j}^{L(t'+t+1)}-h_{i+j}^{L(t'+t)}\bigr)
\eeq
\beq
g_{12}^{L,t_1}(t,j)=\frac{1}{L}\sum_{i=1}^L\frac{1}{t_1}
\sum_{t'=0}^{t_1-1}\bigl(h_{i}^{L(t'+1)}-h_i^{Lt'}\bigr)
\bigl(h_{i+j+1}^{L(t'+t)}-h_{i+j-1}^{L(t'+t)}\bigr)
\eeq
The results are displayed in Fig. 2 and Fig. 3, where the upper indices
$L,t_1$ have been omitted for clarity, while the values of $L$ and
$t_1$ appear in the captions. 

Fig. 2 shows relaxation as function of
time, with both the numerical results as described above, and the
corresponding exact results for the ($oe$: odd-even) sub-lattice
parallel dynamics taken from Prop. \ref{prop1}. The two dynamics
differ for small time but follow similar asymptotics at large
time. The function $g_{12}^{oe}(t,1)$ is not shown on Fig. 2 because,
as mentioned in Remark 1,
it is proportional to $g_{22}^{oe}(t,0)$, 
and would therefore yield the same scaled curve.

\vskip1cm
\centerline{
\setlength{\unitlength}{0.240900pt}
\ifx\plotpoint\undefined\newsavebox{\plotpoint}\fi
\begin{picture}(1500,900)(0,0)
\sbox{\plotpoint}{\rule[-0.200pt]{0.400pt}{0.400pt}}%
\put(140.0,123.0){\rule[-0.200pt]{4.818pt}{0.400pt}}
\put(120,123){\makebox(0,0)[r]{ 0.8}}
\put(1419.0,123.0){\rule[-0.200pt]{4.818pt}{0.400pt}}
\put(140.0,215.0){\rule[-0.200pt]{4.818pt}{0.400pt}}
\put(120,215){\makebox(0,0)[r]{ 0.9}}
\put(1419.0,215.0){\rule[-0.200pt]{4.818pt}{0.400pt}}
\put(140.0,307.0){\rule[-0.200pt]{4.818pt}{0.400pt}}
\put(120,307){\makebox(0,0)[r]{ 1}}
\put(1419.0,307.0){\rule[-0.200pt]{4.818pt}{0.400pt}}
\put(140.0,399.0){\rule[-0.200pt]{4.818pt}{0.400pt}}
\put(120,399){\makebox(0,0)[r]{ 1.1}}
\put(1419.0,399.0){\rule[-0.200pt]{4.818pt}{0.400pt}}
\put(140.0,492.0){\rule[-0.200pt]{4.818pt}{0.400pt}}
\put(120,492){\makebox(0,0)[r]{ 1.2}}
\put(1419.0,492.0){\rule[-0.200pt]{4.818pt}{0.400pt}}
\put(140.0,584.0){\rule[-0.200pt]{4.818pt}{0.400pt}}
\put(120,584){\makebox(0,0)[r]{ 1.3}}
\put(1419.0,584.0){\rule[-0.200pt]{4.818pt}{0.400pt}}
\put(140.0,676.0){\rule[-0.200pt]{4.818pt}{0.400pt}}
\put(120,676){\makebox(0,0)[r]{ 1.4}}
\put(1419.0,676.0){\rule[-0.200pt]{4.818pt}{0.400pt}}
\put(140.0,768.0){\rule[-0.200pt]{4.818pt}{0.400pt}}
\put(120,768){\makebox(0,0)[r]{ 1.5}}
\put(1419.0,768.0){\rule[-0.200pt]{4.818pt}{0.400pt}}
\put(140.0,860.0){\rule[-0.200pt]{4.818pt}{0.400pt}}
\put(120,860){\makebox(0,0)[r]{ 1.6}}
\put(1419.0,860.0){\rule[-0.200pt]{4.818pt}{0.400pt}}
\put(140.0,123.0){\rule[-0.200pt]{0.400pt}{4.818pt}}
\put(140,82){\makebox(0,0){ 1}}
\put(140.0,840.0){\rule[-0.200pt]{0.400pt}{4.818pt}}
\put(284.0,123.0){\rule[-0.200pt]{0.400pt}{4.818pt}}
\put(284,82){\makebox(0,0){ 2}}
\put(284.0,840.0){\rule[-0.200pt]{0.400pt}{4.818pt}}
\put(429.0,123.0){\rule[-0.200pt]{0.400pt}{4.818pt}}
\put(429,82){\makebox(0,0){ 3}}
\put(429.0,840.0){\rule[-0.200pt]{0.400pt}{4.818pt}}
\put(573.0,123.0){\rule[-0.200pt]{0.400pt}{4.818pt}}
\put(573,82){\makebox(0,0){ 4}}
\put(573.0,840.0){\rule[-0.200pt]{0.400pt}{4.818pt}}
\put(717.0,123.0){\rule[-0.200pt]{0.400pt}{4.818pt}}
\put(717,82){\makebox(0,0){ 5}}
\put(717.0,840.0){\rule[-0.200pt]{0.400pt}{4.818pt}}
\put(862.0,123.0){\rule[-0.200pt]{0.400pt}{4.818pt}}
\put(862,82){\makebox(0,0){ 6}}
\put(862.0,840.0){\rule[-0.200pt]{0.400pt}{4.818pt}}
\put(1006.0,123.0){\rule[-0.200pt]{0.400pt}{4.818pt}}
\put(1006,82){\makebox(0,0){ 7}}
\put(1006.0,840.0){\rule[-0.200pt]{0.400pt}{4.818pt}}
\put(1150.0,123.0){\rule[-0.200pt]{0.400pt}{4.818pt}}
\put(1150,82){\makebox(0,0){ 8}}
\put(1150.0,840.0){\rule[-0.200pt]{0.400pt}{4.818pt}}
\put(1295.0,123.0){\rule[-0.200pt]{0.400pt}{4.818pt}}
\put(1295,82){\makebox(0,0){ 9}}
\put(1295.0,840.0){\rule[-0.200pt]{0.400pt}{4.818pt}}
\put(1439.0,123.0){\rule[-0.200pt]{0.400pt}{4.818pt}}
\put(1439,82){\makebox(0,0){ 10}}
\put(1439.0,840.0){\rule[-0.200pt]{0.400pt}{4.818pt}}
\put(140.0,123.0){\rule[-0.200pt]{312.929pt}{0.400pt}}
\put(1439.0,123.0){\rule[-0.200pt]{0.400pt}{177.543pt}}
\put(140.0,860.0){\rule[-0.200pt]{312.929pt}{0.400pt}}
\put(140.0,123.0){\rule[-0.200pt]{0.400pt}{177.543pt}}
\put(789,21){\makebox(0,0){$t$}}
\put(140,307){\usebox{\plotpoint}}
\put(140.0,307.0){\rule[-0.200pt]{312.929pt}{0.400pt}}
\put(1203,768){\makebox(0,0)[r]{
$2^{-3/2}(\pi t)^{1/2}g_{11}(t,0)$}}
\put(140,164){\raisebox{-.8pt}{\makebox(0,0){$\Diamond$}}}
\put(284,241){\raisebox{-.8pt}{\makebox(0,0){$\Diamond$}}}
\put(429,266){\raisebox{-.8pt}{\makebox(0,0){$\Diamond$}}}
\put(573,277){\raisebox{-.8pt}{\makebox(0,0){$\Diamond$}}}
\put(717,283){\raisebox{-.8pt}{\makebox(0,0){$\Diamond$}}}
\put(862,288){\raisebox{-.8pt}{\makebox(0,0){$\Diamond$}}}
\put(1006,291){\raisebox{-.8pt}{\makebox(0,0){$\Diamond$}}}
\put(1150,293){\raisebox{-.8pt}{\makebox(0,0){$\Diamond$}}}
\put(1295,295){\raisebox{-.8pt}{\makebox(0,0){$\Diamond$}}}
\put(1439,296){\raisebox{-.8pt}{\makebox(0,0){$\Diamond$}}}
\put(1273,768){\raisebox{-.8pt}{\makebox(0,0){$\Diamond$}}}
\sbox{\plotpoint}{\rule[-0.400pt]{0.800pt}{0.800pt}}%
\sbox{\plotpoint}{\rule[-0.200pt]{0.400pt}{0.400pt}}%
\put(1203,718){\makebox(0,0)[r]{
$-4(2\pi)^{1/2} t^{3/2}g_{22}(t,0)$}}
\sbox{\plotpoint}{\rule[-0.400pt]{0.800pt}{0.800pt}}%
\put(140,770){\makebox(0,0){$+$}}
\put(284,719){\makebox(0,0){$+$}}
\put(429,548){\makebox(0,0){$+$}}
\put(573,455){\makebox(0,0){$+$}}
\put(717,411){\makebox(0,0){$+$}}
\put(862,385){\makebox(0,0){$+$}}
\put(1006,381){\makebox(0,0){$+$}}
\put(1150,365){\makebox(0,0){$+$}}
\put(1295,348){\makebox(0,0){$+$}}
\put(1439,353){\makebox(0,0){$+$}}
\put(1273,718){\makebox(0,0){$+$}}
\sbox{\plotpoint}{\rule[-0.500pt]{1.000pt}{1.000pt}}%
\sbox{\plotpoint}{\rule[-0.200pt]{0.400pt}{0.400pt}}%
\put(1203,668){\makebox(0,0)[r]{
$(2\pi)^{1/2} t^{3/2}g_{12}(t,1)$}}
\sbox{\plotpoint}{\rule[-0.500pt]{1.000pt}{1.000pt}}%
\put(140,139){\raisebox{-.8pt}{\makebox(0,0){$\Box$}}}
\put(284,365){\raisebox{-.8pt}{\makebox(0,0){$\Box$}}}
\put(429,393){\raisebox{-.8pt}{\makebox(0,0){$\Box$}}}
\put(573,384){\raisebox{-.8pt}{\makebox(0,0){$\Box$}}}
\put(717,373){\raisebox{-.8pt}{\makebox(0,0){$\Box$}}}
\put(862,363){\raisebox{-.8pt}{\makebox(0,0){$\Box$}}}
\put(1006,357){\raisebox{-.8pt}{\makebox(0,0){$\Box$}}}
\put(1150,350){\raisebox{-.8pt}{\makebox(0,0){$\Box$}}}
\put(1295,345){\raisebox{-.8pt}{\makebox(0,0){$\Box$}}}
\put(1439,340){\raisebox{-.8pt}{\makebox(0,0){$\Box$}}}
\put(1273,668){\raisebox{-.8pt}{\makebox(0,0){$\Box$}}}
\sbox{\plotpoint}{\rule[-0.600pt]{1.200pt}{1.200pt}}%
\sbox{\plotpoint}{\rule[-0.200pt]{0.400pt}{0.400pt}}%
\put(1203,568){\makebox(0,0)[r]{
$2^{-1}(\pi t)^{1/2}g_{11}^{oe}(t,0)$}}
\sbox{\plotpoint}{\rule[-0.600pt]{1.200pt}{1.200pt}}%
\put(1223.0,568.0){\rule[-0.600pt]{24.090pt}{1.200pt}}
\put(140,202){\usebox{\plotpoint}}
\multiput(140.00,204.24)(0.732,0.503){6}{\rule{2.250pt}{0.121pt}}
\multiput(140.00,199.51)(8.330,8.000){2}{\rule{1.125pt}{1.200pt}}
\multiput(153.00,212.24)(0.835,0.505){4}{\rule{2.529pt}{0.122pt}}
\multiput(153.00,207.51)(7.752,7.000){2}{\rule{1.264pt}{1.200pt}}
\multiput(166.00,219.24)(0.962,0.509){2}{\rule{2.900pt}{0.123pt}}
\multiput(166.00,214.51)(6.981,6.000){2}{\rule{1.450pt}{1.200pt}}
\put(179,223.01){\rule{3.132pt}{1.200pt}}
\multiput(179.00,220.51)(6.500,5.000){2}{\rule{1.566pt}{1.200pt}}
\put(192,228.01){\rule{3.373pt}{1.200pt}}
\multiput(192.00,225.51)(7.000,5.000){2}{\rule{1.686pt}{1.200pt}}
\put(206,232.51){\rule{3.132pt}{1.200pt}}
\multiput(206.00,230.51)(6.500,4.000){2}{\rule{1.566pt}{1.200pt}}
\put(219,236.01){\rule{3.132pt}{1.200pt}}
\multiput(219.00,234.51)(6.500,3.000){2}{\rule{1.566pt}{1.200pt}}
\put(232,239.51){\rule{3.132pt}{1.200pt}}
\multiput(232.00,237.51)(6.500,4.000){2}{\rule{1.566pt}{1.200pt}}
\put(245,243.01){\rule{3.132pt}{1.200pt}}
\multiput(245.00,241.51)(6.500,3.000){2}{\rule{1.566pt}{1.200pt}}
\put(258,245.51){\rule{3.132pt}{1.200pt}}
\multiput(258.00,244.51)(6.500,2.000){2}{\rule{1.566pt}{1.200pt}}
\put(271,248.01){\rule{3.132pt}{1.200pt}}
\multiput(271.00,246.51)(6.500,3.000){2}{\rule{1.566pt}{1.200pt}}
\put(284,250.51){\rule{3.132pt}{1.200pt}}
\multiput(284.00,249.51)(6.500,2.000){2}{\rule{1.566pt}{1.200pt}}
\put(297,252.51){\rule{3.373pt}{1.200pt}}
\multiput(297.00,251.51)(7.000,2.000){2}{\rule{1.686pt}{1.200pt}}
\put(311,254.51){\rule{3.132pt}{1.200pt}}
\multiput(311.00,253.51)(6.500,2.000){2}{\rule{1.566pt}{1.200pt}}
\put(324,256.51){\rule{3.132pt}{1.200pt}}
\multiput(324.00,255.51)(6.500,2.000){2}{\rule{1.566pt}{1.200pt}}
\put(337,258.51){\rule{3.132pt}{1.200pt}}
\multiput(337.00,257.51)(6.500,2.000){2}{\rule{1.566pt}{1.200pt}}
\put(350,260.01){\rule{3.132pt}{1.200pt}}
\multiput(350.00,259.51)(6.500,1.000){2}{\rule{1.566pt}{1.200pt}}
\put(363,261.51){\rule{3.132pt}{1.200pt}}
\multiput(363.00,260.51)(6.500,2.000){2}{\rule{1.566pt}{1.200pt}}
\put(376,263.01){\rule{3.132pt}{1.200pt}}
\multiput(376.00,262.51)(6.500,1.000){2}{\rule{1.566pt}{1.200pt}}
\put(389,264.01){\rule{3.132pt}{1.200pt}}
\multiput(389.00,263.51)(6.500,1.000){2}{\rule{1.566pt}{1.200pt}}
\put(402,265.51){\rule{3.373pt}{1.200pt}}
\multiput(402.00,264.51)(7.000,2.000){2}{\rule{1.686pt}{1.200pt}}
\put(416,267.01){\rule{3.132pt}{1.200pt}}
\multiput(416.00,266.51)(6.500,1.000){2}{\rule{1.566pt}{1.200pt}}
\put(429,268.01){\rule{3.132pt}{1.200pt}}
\multiput(429.00,267.51)(6.500,1.000){2}{\rule{1.566pt}{1.200pt}}
\put(442,269.01){\rule{3.132pt}{1.200pt}}
\multiput(442.00,268.51)(6.500,1.000){2}{\rule{1.566pt}{1.200pt}}
\put(455,270.01){\rule{3.132pt}{1.200pt}}
\multiput(455.00,269.51)(6.500,1.000){2}{\rule{1.566pt}{1.200pt}}
\put(468,271.01){\rule{3.132pt}{1.200pt}}
\multiput(468.00,270.51)(6.500,1.000){2}{\rule{1.566pt}{1.200pt}}
\put(481,272.01){\rule{3.132pt}{1.200pt}}
\multiput(481.00,271.51)(6.500,1.000){2}{\rule{1.566pt}{1.200pt}}
\put(507,273.01){\rule{3.373pt}{1.200pt}}
\multiput(507.00,272.51)(7.000,1.000){2}{\rule{1.686pt}{1.200pt}}
\put(521,274.01){\rule{3.132pt}{1.200pt}}
\multiput(521.00,273.51)(6.500,1.000){2}{\rule{1.566pt}{1.200pt}}
\put(534,275.01){\rule{3.132pt}{1.200pt}}
\multiput(534.00,274.51)(6.500,1.000){2}{\rule{1.566pt}{1.200pt}}
\put(494.0,275.0){\rule[-0.600pt]{3.132pt}{1.200pt}}
\put(560,276.01){\rule{3.132pt}{1.200pt}}
\multiput(560.00,275.51)(6.500,1.000){2}{\rule{1.566pt}{1.200pt}}
\put(573,277.01){\rule{3.132pt}{1.200pt}}
\multiput(573.00,276.51)(6.500,1.000){2}{\rule{1.566pt}{1.200pt}}
\put(547.0,278.0){\rule[-0.600pt]{3.132pt}{1.200pt}}
\put(599,278.01){\rule{3.132pt}{1.200pt}}
\multiput(599.00,277.51)(6.500,1.000){2}{\rule{1.566pt}{1.200pt}}
\put(586.0,280.0){\rule[-0.600pt]{3.132pt}{1.200pt}}
\put(625,279.01){\rule{3.373pt}{1.200pt}}
\multiput(625.00,278.51)(7.000,1.000){2}{\rule{1.686pt}{1.200pt}}
\put(612.0,281.0){\rule[-0.600pt]{3.132pt}{1.200pt}}
\put(652,280.01){\rule{3.132pt}{1.200pt}}
\multiput(652.00,279.51)(6.500,1.000){2}{\rule{1.566pt}{1.200pt}}
\put(639.0,282.0){\rule[-0.600pt]{3.132pt}{1.200pt}}
\put(678,281.01){\rule{3.132pt}{1.200pt}}
\multiput(678.00,280.51)(6.500,1.000){2}{\rule{1.566pt}{1.200pt}}
\put(665.0,283.0){\rule[-0.600pt]{3.132pt}{1.200pt}}
\put(704,282.01){\rule{3.132pt}{1.200pt}}
\multiput(704.00,281.51)(6.500,1.000){2}{\rule{1.566pt}{1.200pt}}
\put(691.0,284.0){\rule[-0.600pt]{3.132pt}{1.200pt}}
\put(744,283.01){\rule{3.132pt}{1.200pt}}
\multiput(744.00,282.51)(6.500,1.000){2}{\rule{1.566pt}{1.200pt}}
\put(717.0,285.0){\rule[-0.600pt]{6.504pt}{1.200pt}}
\put(783,284.01){\rule{3.132pt}{1.200pt}}
\multiput(783.00,283.51)(6.500,1.000){2}{\rule{1.566pt}{1.200pt}}
\put(757.0,286.0){\rule[-0.600pt]{6.263pt}{1.200pt}}
\put(822,285.01){\rule{3.132pt}{1.200pt}}
\multiput(822.00,284.51)(6.500,1.000){2}{\rule{1.566pt}{1.200pt}}
\put(796.0,287.0){\rule[-0.600pt]{6.263pt}{1.200pt}}
\put(862,286.01){\rule{3.132pt}{1.200pt}}
\multiput(862.00,285.51)(6.500,1.000){2}{\rule{1.566pt}{1.200pt}}
\put(835.0,288.0){\rule[-0.600pt]{6.504pt}{1.200pt}}
\put(914,287.01){\rule{3.132pt}{1.200pt}}
\multiput(914.00,286.51)(6.500,1.000){2}{\rule{1.566pt}{1.200pt}}
\put(875.0,289.0){\rule[-0.600pt]{9.395pt}{1.200pt}}
\put(967,288.01){\rule{3.132pt}{1.200pt}}
\multiput(967.00,287.51)(6.500,1.000){2}{\rule{1.566pt}{1.200pt}}
\put(927.0,290.0){\rule[-0.600pt]{9.636pt}{1.200pt}}
\put(1032,289.01){\rule{3.132pt}{1.200pt}}
\multiput(1032.00,288.51)(6.500,1.000){2}{\rule{1.566pt}{1.200pt}}
\put(980.0,291.0){\rule[-0.600pt]{12.527pt}{1.200pt}}
\put(1111,290.01){\rule{3.132pt}{1.200pt}}
\multiput(1111.00,289.51)(6.500,1.000){2}{\rule{1.566pt}{1.200pt}}
\put(1045.0,292.0){\rule[-0.600pt]{15.899pt}{1.200pt}}
\put(1190,291.01){\rule{3.132pt}{1.200pt}}
\multiput(1190.00,290.51)(6.500,1.000){2}{\rule{1.566pt}{1.200pt}}
\put(1124.0,293.0){\rule[-0.600pt]{15.899pt}{1.200pt}}
\put(1282,292.01){\rule{3.132pt}{1.200pt}}
\multiput(1282.00,291.51)(6.500,1.000){2}{\rule{1.566pt}{1.200pt}}
\put(1203.0,294.0){\rule[-0.600pt]{19.031pt}{1.200pt}}
\put(1400,293.01){\rule{3.132pt}{1.200pt}}
\multiput(1400.00,292.51)(6.500,1.000){2}{\rule{1.566pt}{1.200pt}}
\put(1295.0,295.0){\rule[-0.600pt]{25.294pt}{1.200pt}}
\put(1413.0,296.0){\rule[-0.600pt]{6.263pt}{1.200pt}}
\sbox{\plotpoint}{\rule[-0.500pt]{1.000pt}{1.000pt}}%
\sbox{\plotpoint}{\rule[-0.200pt]{0.400pt}{0.400pt}}%
\put(1203,518){\makebox(0,0)[r]{
$-4(\pi)^{1/2} t^{3/2}g_{22}^{oe}(t,0)$}}
\sbox{\plotpoint}{\rule[-0.500pt]{1.000pt}{1.000pt}}%
\multiput(1223,518)(41.511,0.000){3}{\usebox{\plotpoint}}
\put(161.00,860.00){\usebox{\plotpoint}}
\multiput(166,826)(8.519,-40.628){2}{\usebox{\plotpoint}}
\put(186.01,738.11){\usebox{\plotpoint}}
\put(198.14,698.44){\usebox{\plotpoint}}
\put(212.67,659.57){\usebox{\plotpoint}}
\put(229.40,621.59){\usebox{\plotpoint}}
\put(249.86,585.52){\usebox{\plotpoint}}
\put(274.39,552.08){\usebox{\plotpoint}}
\put(302.87,521.96){\usebox{\plotpoint}}
\put(335.33,496.16){\usebox{\plotpoint}}
\put(370.40,474.02){\usebox{\plotpoint}}
\put(407.58,455.61){\usebox{\plotpoint}}
\put(446.19,440.39){\usebox{\plotpoint}}
\put(485.65,427.57){\usebox{\plotpoint}}
\put(525.74,416.91){\usebox{\plotpoint}}
\put(566.27,408.04){\usebox{\plotpoint}}
\put(607.00,400.15){\usebox{\plotpoint}}
\put(647.98,393.62){\usebox{\plotpoint}}
\put(689.10,388.15){\usebox{\plotpoint}}
\put(730.26,382.96){\usebox{\plotpoint}}
\put(771.54,378.76){\usebox{\plotpoint}}
\put(812.83,374.71){\usebox{\plotpoint}}
\put(854.11,370.61){\usebox{\plotpoint}}
\put(895.49,367.42){\usebox{\plotpoint}}
\put(936.91,365.00){\usebox{\plotpoint}}
\put(978.32,362.13){\usebox{\plotpoint}}
\put(1019.74,359.94){\usebox{\plotpoint}}
\put(1061.17,357.77){\usebox{\plotpoint}}
\put(1102.60,355.65){\usebox{\plotpoint}}
\put(1144.03,353.46){\usebox{\plotpoint}}
\put(1185.49,352.00){\usebox{\plotpoint}}
\put(1226.93,350.16){\usebox{\plotpoint}}
\put(1268.39,348.97){\usebox{\plotpoint}}
\put(1309.86,347.86){\usebox{\plotpoint}}
\put(1351.30,346.00){\usebox{\plotpoint}}
\put(1392.78,345.00){\usebox{\plotpoint}}
\put(1434.25,344.00){\usebox{\plotpoint}}
\put(1439,344){\usebox{\plotpoint}}
\sbox{\plotpoint}{\rule[-0.200pt]{0.400pt}{0.400pt}}%
\put(140.0,123.0){\rule[-0.200pt]{312.929pt}{0.400pt}}
\put(1439.0,123.0){\rule[-0.200pt]{0.400pt}{177.543pt}}
\put(140.0,860.0){\rule[-0.200pt]{312.929pt}{0.400pt}}
\put(140.0,123.0){\rule[-0.200pt]{0.400pt}{177.543pt}}
\end{picture}
}
\noindent\begin{center}{\obeylines\small 
Fig. 2: Random sequential updates, scaled empirical correlation functions $g_{11}(t,0)$, 
$g_{22}(t,0)$ and $g_{12}(t,1)$, average taken over space $L=10^6$ and time $t_1=1000$, and
scaled correlation functions $g_{11}^{oe}(t,0)$, $g_{22}^{oe}(t,0)$ of sub-lattice parallel dynamics.  
}\end{center} 
\bigskip

Fig. 3 shows the variation in space of the space-time correlations at
a given large time $t=10$, together with a fit inspired by 
(\ref{gjj})(\ref{gtt})(\ref{gtj}).

\centerline{\input fig3.tex}
\noindent\begin{center}{\obeylines\small 
Fig. 3: Random sequential updates, scaled empirical correlation functions $g_{11}(t,j)$, 
$g_{22}(t,j)$ and $g_{12}(t,j)$, all at time $t=10$, average taken over space $L=10^6$ and time 
$t_1=1000$, and conjectured asymptotics similar to sub-lattice parallel dynamics.  
}\end{center} 

\newpage
\medskip\noindent{\bf Acknowledgements}

\noindent
The author acknowledges useful discussions with Pablo Ferrari
while visiting IME, Universidade de S\~ao Paulo, at the time when the
present work was being started.

\end{document}